\documentclass[a4paper,10pt,reqno]{amsart}

\usepackage{amsmath, amsthm, amssymb, amsfonts, amscd, mathrsfs, color, relsize, mathtools,microtype}
\newtheorem{theo}{Theorem}[section]
\newtheorem{defi}[theo]{Definition}
\newtheorem{lem}[theo]{Lemma}
\newtheorem{prop}[theo]{Proposition}
\newtheorem{cor}[theo]{Corollary}
\newtheorem{rem}[theo]{Remark}
\newtheorem{exa}[theo]{Example}

\numberwithin{equation}{section}

\newcommand{\w}{\omega}

\newcommand{\R}{\mathbb{R}} 
\newcommand{\C}{\mathbb{C}} 
\newcommand{\Z}{\mathbb{Z}} 
\newcommand{\N}{\mathbb{N}} 

\newcommand{\RR}{\mathfrak{S}}   
\newcommand{\ZZ}{\Delta}      
\newcommand{\ZZdual}{\ZZ^\bot} 
\newcommand{\id}{\mathrm{e}}   

\newcommand{\T}{\mathcal{T}}           
\newcommand{\J}{\mathcal{J}}           
\newcommand{\pj}{\mathcal{P}}          
\newcommand{\A}{\Phi}                  
\newcommand{\Ga}{\Gamma}              
\newcommand{\B}{\mathcal{B}}   
\newcommand{\U}{\mathcal{U}}   
\newcommand{\F}{\mathcal{F}}   
\newcommand{\Hil}{\mathcal{H}} 

\newcommand{\vsp}{\textnormal{span}}
\newcommand{\ol}[1]{\overline{#1}}

\newcommand{\wh}[1]{\widehat{#1}}

\newcommand{\esssup}{{\mathrm{ess}\sup}}

\newcommand{\tr}[1]{T_{#1}}
\newcommand{\rot}[1]{R_{#1}}

\newcommand{\bigpoplus}{\overset{.}{\bigoplus}}
\newcommand{\poplus}{\overset{.}{\oplus}}

\title{The structure of group preserving operators}
\author[Barbieri, Cabrelli, Carbajal, Hern\'andez and Molter]{D. Barbieri, C. Cabrelli, D. Carbajal, E. Hern\'andez and U. Molter}


\begin{document}
	
	\allowdisplaybreaks[2]
	
	\begin{abstract}
	In this paper, we prove the existence of a particular diagonalization for normal bounded operators  defined on subspaces
	of $L^2(\RR)$ where $\RR$ is a second countable LCA group. 
	The subspaces where the operators act are invariant under the action of a group $\Gamma$ which is a semi-direct product of a uniform lattice of $\RR$   with a discrete group of automorphisms. This class includes the crystal groups which are important in applications as models for images. The operators are assumed to be 
	$\Gamma$ preserving. i.e. they commute with the action of $\Gamma$. In particular we obtain a spectral decomposition for these operators.
	
	This  generalizes recent results on shift-preserving operators acting on lattice invariant  subspaces where $\RR$ is the Euclidean space.
	\end{abstract}

	\maketitle
		\section{Introduction}
	
    Let $\Gamma$ be a discrete group, not necessarily commutative, acting on $L^2(\RR)$, where $\RR$ is a second countable LCA group.

In this article, we study the structure of bounded operators acting on  subspaces of $L^2(\RR)$ that are invariant
under the action of the group $\Gamma$ (which we call $\Ga$-invariant subspaces).  Our operators are required to be  $\Gamma$-preserving. This means that they
commute with the action of $\Gamma$. 
	
A recent paper \cite{ACCP}, studied the case where $\RR$ is the $d$-dimensional additive group $\R^d$ and the group $\Ga$ is the lattice $\Delta =\Z^d$  acting by translations on $L^2(\R^d)$.
The authors considered $\Delta$-preserving operators acting on finitely generated $\Delta$-invariant spaces.

  They introduced the notion of $\Delta$-eigenvalue and $\Delta$-diagonalization (see definition in Section \ref{sec:s-diag}) and proved that if $L:V\rightarrow V$ is a bounded normal $\ZZ$-preserving operator defined on a finitely generated $\Delta$-invariant space $V$, then
  $L$ is {\it $\Delta$-diagonalizable}, that is, there exist $r \in \N$ and $\Delta$-invariant subspaces $V_1,\dots,V_r $ such that:
  $$V = V_{1}\poplus\dots\poplus V_{r},$$
  where the sum is orthogonal and the subspaces are invariant under $L$. 
  The action of $L$ on each $f \in V_j$ is given by 
  $Lf = \Lambda_{a_j} f$ with $\Lambda_{a_j}= \sum_{k\in\ZZ} a_{j}(k) T_k$, and $T_k$ denotes the translation by $k\in\ZZ$.
  The operator $\Lambda_{a_j} $ is called a {\it $\Delta$-eigenvalue},
  and is defined by some sequence $a_j \in \ell_2(\ZZ),\; j=1,\dots,r$. As a consequence we have the following formula for $L$:
  $$L=  \sum_{j=1}^r \Lambda_{a_j} P_{V_{j}},$$
  where $P_{V_{j}}$ denotes the orthogonal projection onto $V_{j}$. 
  
  This type of decomposition of a $\ZZ$-preserving operator describes in a simple and compressed way the action of $L$ and is reminiscent of the spectral theorem for normal matrices. The finiteness of the decomposition is a consequence of the fact that the invariant space on which $L$ acts is finitely generated.
  
  In the first part of this paper, we extend this result to the group setting ($\Delta$ is a uniform lattice of a second countable group $\RR$) and we remove the requirement that the invariant space is  finitely generated.  We are able to prove the $\Delta$-diagonalization for normal bounded $\Delta$-preserving operators, mentioned before, satisfying some additional properties.
  
  The main difficulty in pursuing this, is a question about measurability that can not be solved using the arguments of the finitely generated case. We need to resort to the theory of set-valued maps and results on measurable selections, in particular, Castaign's Selection Theorem  (see subsection \ref{MSVM}).
  
  The key tool in the analysis is the characterization of invariant spaces by Helson \cite{Hel1964} through {\it measurable range functions }
  and the decomposition of $\Delta$-preserving operators by {\it measurable range operators} (or direct integrals of operators),
  see \cite{BI2019}.
  
  In the second part of the paper, we extend this decomposition   to $\Ga$-invariant spaces.
  More exactly, we consider a non-commutative group $\Ga$ that is  a semidirect product $\Gamma = \ZZ \rtimes G $. Here, $\ZZ$ is a discrete cocompact subgroup of $\RR$ and $G$ is a discrete and countable group that acts on $\RR$ by continuous automorphisms preserving $\ZZ$. (See Section \ref{sec:Gamma-inv} for details). These groups are important in applications as models for images since they include, as a particular case, rigid movements. See \cite{{BCHMSPIE2019}} and \cite{Mallat2012} for  applications to image processing.
 
 The structure of $\Ga$-invariant spaces has been studied in great detail in a recent paper \cite{BCHM}.
In this article, we consider $\Ga$-preserving operators in this general setting. 
 In order to obtain a diagonalization for these operators we need to define what we mean by $\Ga$-eigenvalues, defined previously for the case of uniform lattices.

Because $\Ga$-invariant subspaces are a particular subclass of $\Delta$-invariant subspaces with extra restrictions,
$\Ga$-eigenvalues should be $\Delta$-eigenvalues with some special property,  due to the action of the group $G$ (see Proposition \ref{prop:Lambda_a-G-preserving}).
Finally, using this, we are able to obtain the desired diagonalization that we call $\Ga$-diagonalization,  	
 
 The paper is organized as follows: Section \ref{largamos} sets the notation that we will use throughout  the paper  and contains all the definitions and properties needed for the diagonalization results. We start in Subsection \ref{sec:SIS} describing the structure of $\Delta$-invariant subspaces and its characterizations through measurable range functions.
 We consider $\Delta$-preserving operators and its associated range operators in Subsection \ref{sec:SP}.
 In Subsection \ref{MSVM}, the definition and basic properties of measurable set-valued maps and a result on measurable selections, are described. Then, we show results on the relationships between  the spectrum of a $\Delta$-preserving
 operator and the spectrum of its range operator in Subsection \ref{spectrum}.
 Sections \ref{sec:s-diag} and \ref{sec:Gamma-inv} contain the main results of this paper. In Section \ref{sec:s-diag} we prove the $\ZZ$-diagonalization in the setting of groups for $\Delta$-invariant spaces not necessarily finitely generated
 and finally in Section \ref{sec:Gamma-inv} we treat the $\Ga$-diagonalization case.

  	\section{Preliminaries}\label{largamos}
   	
	Throughout this work, $\RR$ will be a second countable LCA group and $\ZZ$ will be a uniform lattice of $\RR$ (that is, a  discrete subgroup such that $\RR/\ZZ$ is compact). We will denote by $\wh{\RR}$ the Pontryagin dual of $\RR$ and we will write the characters of $\RR$ indistinctly by
  	$$
  	\langle \xi , x \rangle = e^{2\pi i \xi . x} \ , \quad \xi \in \wh{\RR} , x \in \RR.
  	$$
  	Moreover, the annihilator of $\ZZ$ will be denoted by $$\ZZdual = \{ \ell \in \wh{\RR} \, : \, \langle \ell, k \rangle = 1 \ \forall \, k \in \ZZ\}.$$
  	The Haar measure on $\RR$ of a measurable set $E \subset \RR$ will be denoted by $|E|$. 
  	Furthermore, $\Omega \subset \wh{\RR}$ will always denote a Borel section of $\wh{\RR}/\ZZ^\bot$.
  	
  	We will use the following notation for the Fourier transform of $f \in L^1(\RR)$:
  	$$
  	\F f (\xi) = \wh{f}(\xi) = \int_{\RR} e^{-2\pi i \xi . x} f(x) dx = \int_{\RR} \ol{\langle \xi , x \rangle} f(x) dx,
  	$$
  	which extends by density to an isometric isomorphism in $L^2(\RR)$.
  	
  	We denote by $T : \ZZ \to \U(L^2(\RR))$ the unitary representation defined by 
  	$$
  	\tr{k}f(x) = f(x - k) \ , \quad f \in L^2(\RR), \,k\in \ZZ.
  	$$
  	Note that for all $f \in L^2(\RR)$ and all $k \in \ZZ$, the following relation holds:
  	\begin{equation}\label{relations-tau-Tk}
  	\wh{\tr{k} f}(\xi) = e^{-2\pi i \xi . k} \wh{f}(\xi), \quad \xi\in\wh{\RR}.
  	\end{equation}
  	
  	If $\Hil$ is a Hilbert space, we  denote by $\B(\Hil)$ the linear and  bounded operators from $\Hil$ into $\Hil$. Given an operator $A\in\mathcal B(\Hil)$ we  denote by $\sigma(A)$ its operator spectrum. The point spectrum of $A$, that is the set of its eigenvalues, is denoted by $\sigma_p(A)$.
  	
  	A normal operator $A\in\mathcal B(\Hil)$ is called diagonalizable if $\Hil$ admits an orthonormal basis consisting of eigenvectors of $A$. We will use the symbol $\poplus$ to denote an orthogonal sum. 
  	
  	\subsection{$\ZZ$-invariant spaces} \label{sec:SIS}
		
	We begin by introducing some important notions on $\ZZ$-invariant spaces, also known as shift-invariant spaces by translations of $\ZZ$.
	\begin{defi}
		A closed subspace $V \subset L^2(\RR)$ is $\ZZ$-invariant if $\tr{k}V \subset V$ for all $k \in \ZZ$.
	\end{defi}
	
	Given a countable set of functions $\A \subset L^2(\RR)$, we will denote 
	$$
	S(\A):= \ol{\vsp}\{\tr{k}\varphi \, : \, k \in \ZZ, \varphi \in \A\}\,.
	$$
	Since $L^2(\RR)$ is separable, if $V$ is a $\ZZ$-invariant subspace of $L^2(\RR)$, there exists a countable set $\A \subset L^2(\RR) $ such that $V = S(\A).$ In this case, we say that $\A$ is a set of generators of $V$. Moreover, if $V$ admits a finite set of generators, we say that $V$ is finitely generated, and if $V=S(\varphi)$ for $\varphi \in L^2(\RR)$ we say that $V$ is principal. 
	
	\begin{defi}\label{def:mappingT}
		For $f\in L^2(\RR)$ and $\omega \in \wh{\RR}$ define formally the fiberization map $\T$ as
		\begin{equation}\label{eqn:mappingT}
		\T[f] (\omega) = \{\wh{f}(\omega + \ell)\}_{\ell \in \ZZdual}.
		\end{equation}
	\end{defi}
	
	The fiberization map  $\T$ is an isometric isomorphism between the Hilbert spaces $L^2(\RR)$ and $ L^2(\Omega,\ell_2(\ZZdual))$, see \cite[Proposition 3.3]{CP2010}. Observe that by \eqref{relations-tau-Tk}, for every $f\in L^2(\RR)$ and $k\in\ZZ$, we have the following intertwining property:
	\begin{equation}\label{eq:Tau-intertwining}
		\T[T_k f](\w)=e^{-2\pi i \w .k} \T[f](\w),\quad \w\in\Omega.
	\end{equation}
		
	The following map, first introduced by Helson \cite{Hel1964}, is fundamental in the theory of shift-invariant spaces.
	
	\begin{defi}\label{def:range-function}
		A range function is a map
		$$
		\J : \Omega \to \{\textnormal{closed subspaces of } \ell_2(\ZZdual)\} .
		$$
	\end{defi}
	
	We say that a range function $\J$ is {\it measurable} if $\omega \mapsto \langle \pj_{\J(\omega)}a,b\rangle_{\ell_2(\ZZdual)}$ is measurable for all $a, b \in \ell_2(\ZZdual)$, where $\pj_{\J(\omega)} \in \B(\ell_2(\ZZdual))$ is the orthogonal projection onto $\J(\omega)$. 
	
	The next theorem is due to Helson \cite{Hel1964} and Bownik \cite{Bow2000} in the Euclidean setting. We state its generalization to the setting of LCA groups as it appears in \cite{CP2010}.
	
	\begin{theo}[\cite{Bow2000,CP2010}]\label{Th:Helson}
		Let $V$ be a closed subspace of $L^2(\RR)$ and $\T$ the map of Definition \ref{def:mappingT}. The subspace $V$ is $\ZZ$-invariant if and only if there exists a unique measurable range function $\J_V$ such that 
		$$V=\left\{f\in L^2(\RR)\,:\, \T[f](\w)\in \J_V(\w),\,a.e. \ \w\in\Omega\right\}.$$
		
		Moreover, if $V=S(\A)$ for some countable set $\A$ of $L^2(\RR)$, the measurable range function associated to $S(\A)$ satisfies
		$$
		\J_V(\omega) =\ol{\vsp}\{\T[\varphi](\omega) \, : \, \varphi \in \A\}, \quad a.e. \ \omega \in \Omega.
		$$
	\end{theo}
	
	\begin{rem}\label{range-unicidad}
		Uniqueness of the the measurable range function is understood in the following sense: two range functions $\J_1$ and $\J_2$ are equal if $\J_1(\omega) = \J_2(\omega)\  a. e. \ \omega \in \Omega.$ 
	\end{rem}
	
	From now on, given a $\ZZ$-invariant space $V$, we will simply write its associated range function as $\J$ when it is clear from the context that we are referring to $\J_V$.
	
	Given a range function $\J$, the  space 
	\begin{equation}\label{eq:MJ}
	\mathcal M_{\J} = \{F\in L^{2}\left(\Omega,\ell^{2}(\ZZdual)\right) : F(\omega)\in \J(\omega),\,\text{for a.e. }\w\in\Omega \},
	\end{equation}
	is a closed {\it multiplicative-invariant} subspace of $L^2\left(\Omega,\ell_2(\ZZdual)\right)$, i.e. for every $F \in \mathcal M_{\J}$ we have 	that $\psi F \in \mathcal M_{\J}$ for all $\psi \in L^\infty(\Omega)$. By Theorem \ref{Th:Helson}, if $V$ is a $\ZZ$-invariant space with 		range function $\J$, we have that $\T[V]= \mathcal M_{\J}$.
	
	The following identity is due to Helson \cite{Hel1964},
	\begin{equation}\label{eq:proj}
	(P_{\mathcal M_\J}F)(\w) = P_{\J(\w)}(F(\w)),\quad\forall\, F\in L^2(\Omega,\ell_2(\ZZdual)), \text{ a.e. }\w\in\Omega,
	\end{equation}
	and as a consequence, the next proposition holds.
	
	\begin{prop}\label{prop:perp}
		Let $V\subset L^2(\RR)$ be a $\ZZ$-invariant space with range function $\J_V$. Then $V^\perp$ is also a $\ZZ$-invariant space with range function 
		$$\J_{V^\perp}(\w)=(\J_V(\w))^\perp, \text{ a.e. }\w\in\Omega.$$
	\end{prop}

	The result below gives a characterization of frames of translations of a $\ZZ$-invariant space $V$ in terms of its fibers, see \cite{Bow2000,CP2010}.

	\begin{theo}[\cite{Bow2000,CP2010}]\label{thm:frame-fiber}
		Let $\Phi\subset L^2(\RR)$ be a countable set. Then the following conditions are equivalent:
		\begin{enumerate}
			\item The system $\{\tr{k}\varphi \, : \, k \in \ZZ, \varphi \in \A\}$ is a frame of $V$ with bounds $A,B>0$.
			\item The system $\left\{\mathcal T [\varphi](\omega)\,:\,\varphi\in\Phi\,\right\}\subset\ell_2(\ZZdual)$ is a frame of $\J(\omega)$ with uniform bounds $A,B>0$ for a.e. $\omega\in\Omega$.
		\end{enumerate}
	\end{theo}

	
	\begin{defi}
		The {\it spectrum} of a $\ZZ$-invariant space $V$ with range function $\J$ is defined by $$ \Sigma(V) = \left\{\omega \in \Omega\,:\, \dim \J(\omega) > 0 \right\}.$$
	\end{defi}

	The result we state next gives a decomposition of a $\ZZ$-invariant space into an orthogonal sum of principal $\ZZ$-invariant spaces satisfying some additional properties. In \cite{Bow2000} the theorem was proved in the Euclidean case but the proof can be extended in a straightforward manner to our setting. 
	
	\begin{theo}[{\cite[Theorem 3.3]{Bow2000}}]\label{thm:principal-decomposition}
		Let $V$ be a $\ZZ$-invariant space of $L^2(\RR)$. Then $V$ can be decomposed as an orthogonal sum
		\begin{equation}\label{principal-decomposition}
		V=\underset{i\in\N}{\bigpoplus} S(\varphi_i),
		\end{equation}
		where $\varphi_i$ is a Parseval frame generator of $S(\varphi_i)$, and $\Sigma(S(\varphi_{i+1}))\subset \Sigma(S(\varphi_i))$ for all $i\in\N$.
	\end{theo}

	As a consequence, one obtains the following lemma which for the case of  $\dim \J(\w)<\infty$ for a.e. $\w\in\Omega$ has been proved in \cite{ACCP}, and extends with little effort to the general case. 
		
	\begin{lem}\label{lem:An}
		Let $V$ be a $\ZZ$-invariant space with range function $\J$. Then, there exist functions $\{\varphi_i\}_{i\in\N}$ of $L^2(\RR)$ and a family of disjoint measurable sets $\{A_n\}_{n\in\N_0}$ and $A_{\infty}$, such that $\Omega=\left( \bigcup_{n\in\N_0} A_n\right) \cup A_{\infty}$ and the following statements hold:
		\begin{enumerate}
			\item $\{T_k\varphi_i\,:\,i\in\N,\, k\in\ZZ\}$ is a Parseval frame of $V$,
			\item for every $n\in \N$ and for every $i>n$, $\T\varphi_i(\w)=0$ a.e. $\w\in A_n$,
			\item  for every $n\in \N$, $\{\T\varphi_1(\w),\dots,\T\varphi_n(\w)\}$ is an orthonormal basis of $\J(\w)$ for a.e. $\w\in A_n$,
			and $\{\T\varphi_i(\w)\}_{i\in\N}$ is an orthonormal basis of $\J(\w)$ for a.e. $\w\in A_\infty$,
			\item for every $n\in\N_0$, $\dim \J(\w) = n$ for a.e. $\w\in A_n$ and $\dim \J(\w)=\infty$ for a.e. $\w\in A_{\infty}$.
		\end{enumerate}
	\end{lem}

%

	\subsection{$\ZZ$-preserving operators}\label{sec:SP} The natural operators acting on $\ZZ$-invariant spaces are the $\ZZ$-preserving operators. These are the ones that commute with translations by elements of $\ZZ$, also known in the literature as shift-preserving operators.
	
	\begin{defi}
		Let $V, V'\subset L^2(\RR)$ be two $\ZZ$-invariant spaces and let $L:V\rightarrow V'$ be a bounded operator. We say that $L$ is $\ZZ$-preserving if $LT_k = T_k L$ for all $k\in\ZZ$.
	\end{defi}

	The structure of $\ZZ$-preserving operators can be understood through the concept of its range operator, which was first introduced in the Euclidean context in \cite{Bow2000}.
	
	\begin{defi}
		Given measurable range functions
		\begin{equation*}
		\J,\J':\Omega\rightarrow \{\text{closed subspaces of }\ell_2(\ZZdual)\},
		\end{equation*}
		a range operator $O:\J\rightarrow \J'$ is a choice of linear operators $O(\w):\J(\w)\rightarrow \J'(\w)$, $\w\in\Omega$.

		A range operator $O$ is said to be bounded if $\esssup_{\w\in\Omega} \|O(\w)\|_{\text{op}} < \infty$, and
		is measurable if $\w\mapsto\langle O(\w)P_{\J(\w)}a,b\rangle_{\ell_2(\ZZdual)}$ is measurable for all $a,b\in\ell_2(\ZZdual)$.
	\end{defi}
	There exists a correspondence between bounded $\ZZ$-preserving operators and bounded measurable range operators. In what follows we describe how this correspondence can be  deduced.  
	
	\begin{defi}\label{def:M_psi}
		For $\psi\in L^\infty(\Omega)$, denote as
		$M_{\psi}:L^2\left(\Omega,\ell_2(\ZZdual)\right)\rightarrow  L^2\left(\Omega,\ell_2(\ZZdual)\right)$ the multiplication operator 
		\begin{equation*}
		M_\psi F (\w) = \psi(\w)F(\w),\quad F \in L^2\left(\Omega,\ell_2(\ZZdual)\right),\,\w\in\Omega,
		\end{equation*}
		which is well defined and bounded.
	\end{defi}
	
	Let $D=\{e^{-2\pi i \w . k}\}_{k\in\ZZ}\subseteq L^\infty(\Omega)$, then $D$ is a determining set for $L^1(\Omega)$ (see \cite[Definition 3.3]{BI2019}). If $L:V\rightarrow V'$ is a bounded $\ZZ$-preserving operator, then the operator 
	\begin{equation}\label{eq:Ltilde}
	\tilde{L}=\T L \T^{-1} : \mathcal M_{\J_V} \rightarrow \mathcal M_{\J_{V'}}
	\end{equation} is bounded and, by \eqref{eq:Tau-intertwining}, satisfies that $\tilde{L} M_{\psi} = M_{\psi} \tilde{L}$ for every $\psi \in D$. By \cite[Theorem 3.7]{BI2019}, there exists a bounded measurable range operator $O:\J_V\rightarrow \J_{V'}$ such that $\tilde{L}F(\w) = O(\w)F(\w)$, for every $F \in \mathcal M_{\J_V}$, $\w\in \Omega$. That is, 
	\begin{equation}
	\T [Lf](\w) = O(\w) \T[f] (\w), \quad f\in V,\,\w\in\Omega.
	\end{equation}
	Moreover, the correspondence between $L$ and $O$ is one-to-one if we identify range operators that agree a.e. $\w\in \Omega$.
	
	A different way to understand range operators is through the direct integral theory. In fact, it can be proved (see \cite{BI2019}) that $$\mathcal M_{\J_V} = \int_{\Omega}^\oplus \J_V(\w)\,d\w$$ and that the operator $\tilde{L}$ defined in \eqref{eq:Ltilde} is a decomposable operator such that $$\tilde{L}=\int_{\Omega}^{\oplus} O(\w)\,d\w.$$
	
	In the following theorem we enumerate some results that relate the properties of $L$ with the pointwise properties of its range operator $O$ (see \cite{BI2019} for proofs).
	
	\begin{theo}[\cite{BI2019}]\label{thm:pointwise}
		Let $V,V'\subset L^2(\RR)$ be two $\ZZ$-invariant spaces. Let $L:V\rightarrow V'$ be a $\ZZ$-preserving operator with corresponding range operator $O:\J_V\rightarrow \J_{V'}$. Then the following are true:
		\begin{enumerate}
			\item\label{op-norm} $\|L\|_{\text{op}}=\esssup_{\w\in\Omega} \|O(\w)\|_{\text{op}}$.
			\item The adjoint $L^*:V'\rightarrow V$ is also $\ZZ$-preserving with corresponding range operator $O^*:\J_{V'}\rightarrow \J_V$ given by $O^*(\w)=(O(\w))^*$ for a.e. $\w\in\Omega$.
			\item $L$ is normal (self-adjoint) if and only if $O(\w)$ is normal (self-adjoint) for a.e. $\w\in\Omega$.
			\item $L$ is injective if and only if $O(\w)$ is injective for a.e. $\w\in\Omega$.
			\item\label{partial-isom} $L$ is a (partial) isometry if and only if $O(\w)$ is a (partial) isometry for a.e. $\w\in\Omega$.
			\item\label{rank} The space $V''=\overline{L(V)}\subseteq L^2(\RR)$ is $\ZZ$-invariant and its range function is given by $$\J_{V''}(\w)=\overline{O(\w) \J_V(\w)},$$
			for a.e. $\w\in\Omega$.
			\item\label{ker} The space $\ker(L)$ is $\ZZ$-invariant and its range function is given by $K(\w)=\ker(O(\w))$ for a.e. $\w\in\Omega$.
		\end{enumerate}
	\end{theo} 
	
\subsection{Measurable set-valued maps}\label{MSVM}

Given $L:V\rightarrow V$ a bounded $\ZZ$-preserving operator, there is a relation between the spectrum of $L$ and the pointwise spectrum of its range operator, as we will discuss in the next subsection. For that, we need to introduce the definition of measurable set-valued maps. We refer the reader to \cite{AF2009} for a detailed exposition on the set-valued maps' theory. 

\begin{defi}
	Let $(X,\mathcal M)$ be a measurable space and $Y$ a topological space. A set-valued map  from $X$ to $Y$ is a map $F:X\rightsquigarrow Y$ whose values are sets in $Y$. That is, $F(x)\subseteq Y$ for every $x\in X$. If $F(x)$ is closed (compact) for every $x\in X$, then $F$ is said to be a {\em set-valued map to closed (compact) values}.
	
	A set-valued map is said to be measurable if for every open set $O\subset Y$, the set
	$$F^{-1}(O) := \{x\in X\,:\, F(x)\cap O \neq \emptyset\}\in\mathcal M.$$
\end{defi}

For example, in \cite{BI2019} it was proved that a measurable range function $\J$ is a measurable set-valued map $\Omega\rightsquigarrow \ell_2(\ZZdual)$ to non-empty closed values.

One very important result that we will need later is the existence of a dense set of measurable selections for measurable set-valued maps, which is known as Castaign's Selection Theorem (see \cite{AF2009} for a proof). 

\begin{defi}
	Let $(X,\mathcal M)$ be a measurable space and $Y$ a topological space. Given $F:X\rightsquigarrow Y$ a measurable set-valued map, we say that a measurable function $f:X\to Y$ is a measurable selection of $F$ if  $f(x)\in F(x)$ for every $x\in X$.
\end{defi}

\begin{theo}[Castaign's Selection Theorem] \label{thm:castaign}
	Let $(X,\mathcal M)$ be a measurable space, $Y$ a complete separable metric space and $F:X\rightsquigarrow Y$ a measurable set-valued map to non-empty closed values, then there exists a sequence of measurable selections $f_j:X\to Y$, $j\in\N$ such that  for every $x\in X$.
	\begin{equation}
	F(x)=\overline{\{f_j(x)\,:\,j\in\N\}}.
	\end{equation}
	
\end{theo}

\subsection{The spectrum of $\ZZ$-preserving operators}\label{spectrum}
	
We start this subsection with Theorem \ref{thm:measurable-spectrum} and Theorem \ref{thm:spectral-measure} whose proofs appeared in \cite{BI2019} and also in \cite{Chow1970, Len1974} in the context of decomposable operators on direct integral Hilbert spaces.

	The first theorem establishes  that the spectra of the fibers of a $\ZZ$-preserving operator $L$ define a measurable set-valued map and describe the relationship between those spectra and the spectrum of $L$. 
	
	\begin{theo}[\cite{BI2019}]\label{thm:measurable-spectrum}
		Let $L:V\rightarrow V$ be a $\ZZ$-preserving operator with range operator $O:\J\rightarrow \J$. Then  $F:\Omega\rightsquigarrow \C$ defined by $F(\w)=\sigma(O(\w))$, $\w\in \Omega$ is a measurable set-valued map to non-empty compact values and $F(\w) \subseteq \sigma(L)$ for a.e. $\w\in\Omega$.
		
		Moreover, when $L$ is normal, $\sigma(L)$ coincides with the smallest closed subset of $\C$ that contains $F(\w)$ for a.e. $\w\in\Omega$. 
	\end{theo}
	
	Suppose now that $L:V\rightarrow V$ is a normal, bounded and $\ZZ$-preserving operator. Then, there exists a spectral measure $E$ of $L$ and we have that $$L=\int_{\sigma(L)} \lambda\, dE(\lambda).$$ Since the range operator $O$ satisfies that $O(\w)$ is normal for a.e. $\w\in\Omega$, then there exists a spectral measure $E_\w$ of $O(\w)$ for a.e. $\w\in\Omega$ and $$O(\w)=\int_{\sigma(O(\w))} \lambda \,dE_\w(\lambda).$$ In this direction, the following result was obtained.
	
	\begin{theo}[\cite{BI2019}]\label{thm:spectral-measure}
		Let $L:V\rightarrow V$ a normal, bounded and $\ZZ$-preserving operator with range operator $O$. Let $E$ be the spectral measure of $L$ and $E_\w$ the spectral measure of $O(\w)$ for a.e. $\w\in\Omega$. Then, for any Borel set $B\subset\C$, $E(B)$ is a $\ZZ$-preserving operator and its range operator is given by $E_\w(B)$.
	\end{theo}

		In the next section, we will discuss to what extent the diagonalization properties of the range operator can provide the $\ZZ$-preserving operator with a special structure. For this purpose, we are interested in finding measurable selections of the eigenvalues of the range operator.  
		
		Assume first that $L$ is acting on a $\ZZ$-invariant space $V$ such that $\dim \J(\w)<\infty$ for a.e. $\w\in\Omega$. Then, we have that $O(\w):\J(\w)\to\J(\w)$ is an operator acting on a finite-dimensional space for a.e. $\w\in\Omega$. 
		In \cite{ACCP}, a construction of a measurable selection of the eigenvalues of $O$ was obtained in the following sense.
	
	\begin{theo}[\cite{ACCP}]\label{thm:measurable-eigen}
		Let $O:\J\rightarrow \J$ be a bounded measurable range operator on a range function satisfying $\dim \J(\w)<\infty$ for a.e. $\w\in\Omega$. Then, there exist  functions $\lambda_j\in L^\infty(\Omega) $, $j\in\N$, such that
		\begin{enumerate}
			\item  $\lambda_j(\w)\neq\lambda_{j'}(\w)$ for $j\neq j'$ and for a.e. $\w\in\Omega$,
			\item $\sigma(O(\w))=\{\lambda_1(\w),\dots,\lambda_i(\omega)\}$ for a.e. $\w\in A_{n,i}$ and for every $i\leq n$, $i,n \in \N$,
		\end{enumerate} 
		where $A_{n,i}$ are the measurable sets:
		\begin{equation}
		A_{n,i}:= \left\{\w\in A_n\,:\, \#\sigma(O(\w)) = i\right\},
		\end{equation}
		and $\{A_n\}_{n\in\N}$ are the sets defined in Lemma \ref{lem:An}.
	\end{theo}

	\begin{rem}\label{rem:esssup-eigen}
		If $r=\esssup_{\w\in\Omega}\,\# \sigma(O(\w))<\infty$, then $|A_{n,i}|=0$ for every $i> r$. Thus if we discard the	functions $\lambda_j$ such that $\lambda_j(\w)$  is not an eigenvalue of $O(\w)$ for a.e. $\w\in\Omega$, the number of measurable functions constructed, after discarding, will be $r$ in total. 
	\end{rem}

	We remark that $\dim \J(\w)<\infty$ for a.e. $\w\in\Omega$ does not imply that $V$ is finitely generated, as we show in the example below.
	
	\begin{exa}
		Let $\RR=\R$, $\ZZ=\Z$ and $\Omega=[0,1)$. For every $n\in \N$, define the set $E_n := [0,\frac{1}{n})$ and $\psi_n\in L^2(\R)$ by the map $\T$ as $\T[\psi_n]:=e_n\chi_{E_n}$, where $e_n$ is the nth canonical sequence of $\ell_2(\Z)$. Then $V:=S(\{\psi_n\,:\,n\in\N\})$ is a $\Z$-invariant space which is not finitely generated and its range function $\J$ satisfy that $\dim \J(\w)<\infty$ for a.e. $\w\in [0,1)$. Indeed, let $A_n:=E_n\setminus E_{n+1}$ for $n\in\N$, then $[0,1)= \bigcup_{n\in\N} A_n \cup \{0\}$. For every $n\in\N$ and for a.e. $\w\in A_n$, the dimension of $\J(\w)$ is $n$ since $\{\T[\psi_1](\w),\dots,\T[\psi_n](\w)\}$ is an orthonormal basis of $\J(\w)$.
	\end{exa}

	To remove the condition $\dim \J(\w)<\infty$ for a.e. $\w\in\Omega$, Theorem \ref{thm:measurable-eigen} is no longer useful since its proof strongly relies on the fact that the dimension of $\J(\w)$ is finite for a.e. $\w\in \Omega$. In the following section we will see that, under certain conditions, Castaign's Selection Theorem (Theorem \ref{thm:castaign}) will be helpful in this endeavor.

	\section{$\ZZ$-diagonalization}\label{sec:s-diag}
	
	We are interested in studying the structure of bounded, normal and $\ZZ$-preserving operators whose fibers are diagonalizable operators almost everywhere.
	The question that arises is the following. Suppose that $L:V\to V$ is a bounded, normal and $\ZZ$-preserving operator with range operator $O:\J\to\J$. If $O(\w)$ is diagonalizable for a.e. $\w\in \Omega$, does this induce any simpler kind of decomposition for $L$? 
	
	This question has been studied in \cite{ACCP} in the Euclidean setting with $\ZZ=\Z^d$, where a positive answer was obtained for the case of normal $\ZZ$-preserving operators acting on finitely generated $\ZZ$-invariant spaces. For this, the authors introduced three new concepts which they called $s$-eigenvalue, $s$-eigenspace and $s$-diagonalization. 
	
	\subsection{Background}
	
	In this subsection we will review part of the work done in \cite{ACCP}. Along the way, we will translate the statements to the general group setting and we will show the results that can be effortlessly extended to $\ZZ$-invariant spaces that are not finitely generated. 
	

	Given a sequence $a=\{a(s)\}_{s\in\ZZ}\in\ell_1(\ZZ)$, we denote its Fourier transform as
	$$\wh{a}(\w)=\sum_{s\in\ZZ}a(s)e^{-2\pi i \w. s},\quad \w\in\Omega.$$ 
	This extends to $\ell_2(\ZZ)$ and it holds that  $a\in\ell_2(\ZZ)$ if and only if \ $\wh{a} \in L^2(\Omega)$. Moreover, if \ $\wh{a}\in L^{\infty}(\Omega)$, we will say that $a$ is of bounded spectrum.


%
%

	\begin{defi}\label{def:Lambda_a}
		Given $a\in\ell_2(\ZZ)$ of bounded spectrum, let $M_{\hat{a}}:L^2(\Omega,\ell_2(\ZZdual))\to L^2(\Omega,\ell_2(\ZZdual))$ be the multiplication operator by \ $\wh{a}$, as in Definition \ref{def:M_psi}. We denote by $\Lambda_a:L^2(\RR)\to L^2(\RR)$ the operator defined by $$\Lambda_a:=\T^{-1} M_{\hat{a}} \T,$$ 
		which is clearly well-defined and bounded. 
	\end{defi}

	Let us denote by $\mathfrak{B}$ the following set:
	$$\mathfrak B  :=\{\varphi\in L^2(\RR)\,:\, \{T_k\varphi\}_{k\in\ZZ} \text{ is a Bessel sequence}\}.$$
	Recall that if $a\in\ell_2(\ZZ)$ and $f\in \mathfrak B$, then the series \begin{equation}\label{eq:series-Ts}
		\sum_{s\in\ZZ}a(s)T_s f
	\end{equation} converges in $L^2(\RR)$. 
	

	\begin{prop}
		Let $a=\{a(s)\}_{s\in \ZZ}\in\ell_2(\ZZ)$ be of bounded spectrum. 
		If $f\in\mathfrak B$, then $$\Lambda_a f =\sum_{s\in\ZZ}a(s)T_sf,$$
		with convergence in $L^2(\RR)$.
	\end{prop}
	
	\begin{proof} Given that \eqref{eq:series-Ts} is convergent in $L^2(\RR)$, for a.e. $\w\in\Omega$,
		\begin{align*}
		\T\left(\sum_{s\in\ZZ}a(s)T_s f\right)(\w) & = \left\{\mathcal F\left(\sum_{s\in\ZZ}a(s)T_s f\right)(\w + \ell)\right\}_{\ell\in\ZZdual}\\
		 &= \left\{\sum_{s\in\ZZ}a(s)e^{-2\pi i\w. s } \wh{f}(\w + \ell)\right\}_{\ell\in\ZZdual}\\
		 &=  \wh{a}(\w)\{\wh{f}(\w+\ell)\}_{\ell\in \ZZdual}\\
		 & =M_{\hat{a}}\mathcal T f(\w).
		\end{align*}
		Thus, $\Lambda_af  = \sum_{s\in\ZZ}a(s)T_s f$.
	\end{proof}

	On the other hand, $\mathfrak B$ is a dense set of $L^2(\RR)$ since the functions of compact support in $L^2(\RR)$ belong to $\mathfrak B$ (see \cite[Proposition 9.3.4]{Chr2003}, where a proof is given in the Euclidean case and can be immediately extended to our group context). Then, if $a\in\ell_2(\ZZ)$ is of bounded spectrum, it is possible to give an alternative definition for $\Lambda_a$ as the continuous extension of the bounded operator $\tilde\Lambda_a:\mathfrak{B}\to L^2(\RR)$, defined by 
	\begin{equation}\label{eq:tilde-Lambda_a}
		\tilde\Lambda_a f := \sum_{s\in\ZZ}a(s)T_sf.
	\end{equation}
	For this reason, sometimes we will write $\Lambda_a f$ as the series  \eqref{eq:tilde-Lambda_a}, even for functions which are not in $\mathfrak{B}$, meaning the extension of $\tilde\Lambda_a$ to $L^2(\RR)$.
	
	Notice that in the particular case when $a\in\ell_1(\ZZ)$ (and thus of bounded spectrum), an easy computation shows that $\Lambda_a=\sum_{s\in\ZZ}a(s)T_s$ where the convergence is in the strong operator topology of $\mathcal B(L^2(\RR))$.

	
	
	Observe that if $V$ is $\ZZ$-invariant, then $\Lambda_a(V)\subseteq V$ whenever it is bounded, and in this case $\Lambda_a: V\to V$  is a  $\ZZ$-preserving  operator with corresponding range operator $O_a(\omega)=\wh{a}(\omega)\mathcal I_{\w}$, a.e. $\omega\in\Omega$, where $\mathcal I_{\w}$ denotes the identity operator on $\J(\w)$ for a.e. $\w\in\Omega$.
	
	The following  corresponds to the definition of $s$-eigenvalue and $s$-eigenvector in \cite{ACCP}.
	
	\begin{defi}
		Let $V \subset L^2(\RR)$ be a $\ZZ$-invariant space and $L:V\rightarrow V$ a bounded $\ZZ$-preserving operator. Given $a\in \ell_{2}(\ZZ)$ a sequence of bounded spectrum, we say that $\Lambda_a$ is a $\ZZ$-eigenvalue of $L$ if
		\begin{equation*} 
		V_a := \ker\left(L - \Lambda_a\right)\neq\{0\}.
		\end{equation*}
		We call $V_a$ the $\ZZ$-eigenspace associated to $\Lambda_a$.
	\end{defi} 
	
	These $\ZZ$-eigenspaces $V_a$ are $\ZZ$-invariant spaces and satisfy that $LV_a\subseteq V_a.$  The proposition below was proved in  \cite{ACCP} showing  that the $\ZZ$-eigenvalues of $L$ are intrinsically related to the eigenvalues of the range operator of $L$.
	
	\begin{prop}\label{prop:eigen}
		Let $V \subset L^2(\RR)$ be a $\ZZ$-invariant space with range function $\J$, $L:V\rightarrow V$ a bounded $\ZZ$-preserving operator with range operator $O$ and  $a\in \ell_{2}(\ZZ)$ a sequence of bounded spectrum. Then, the following statements hold:
		\begin{enumerate}
			\item\label{prop-eigen-1} If $\Lambda_a$ is a $\ZZ$-eigenvalue of $L$, then $\wh{a}(\omega)$ is an eigenvalue of $O(\omega)$ for a.e. $\omega\in \Sigma(V_a)$.
			\item  The mapping $\omega \mapsto \ker\left(O(\omega) - \wh{a}(\omega)\mathcal I_{\w}\right)$, $\omega \in \Omega$ is the measurable range function of $V_a$, which we will denote by $\J_{a}$.
		\end{enumerate}
	\end{prop}

	\begin{rem}
		In fact, the converse for statement \eqref{prop-eigen-1} in Proposition \ref{prop:eigen} is true in the following sense: if \ $\wh{a}(\w)$ is an eigenvalue of $O(\w)$ for a.e. $\w$ in a set of positive measure, then $\Lambda_a$ is a $\ZZ$-eigenvalue of $L$.
	\end{rem}

	The following is an extension of the definition of $s$-diagonalization given in \cite{ACCP}, which was originally stated for finitely generated $\ZZ$-invariant spaces. Here, we extend the definition to any $\ZZ$-invariant space.
	
	\begin{defi}\label{def:s-diag}
		Let $V \subset L^2(\RR)$ be a $\ZZ$-invariant space and  $L:V\rightarrow V$ a bounded, $\ZZ$-preserving operator. We say that $L$ is $\ZZ$-diagonalizable if there exists a set of sequences of bounded spectrum $\{a_j\}_{j\in I}\subseteq \ell_2(\ZZ)$, where $I$ is at most countable, such that $\Lambda_{a_j}$ is a $\ZZ$-eigenvalue of $L$ for every $j\in I$ and $V$ can be decomposed into the orthogonal sum
		\begin{equation}\label{direct sum V}
		V = \underset{j\in I}{\bigpoplus} V_{a_j}.
		\end{equation}
		Given such a decomposition, we will say that $\{a_j\}_{j\in I}\subseteq \ell_2(\ZZ)$ is a $\ZZ$-diagonalization of $L$.
	\end{defi}
	
	If an operator $L$ is $\ZZ$-diagonalizable, a decomposition as in \eqref{direct sum V} exists but is not unique. Observe that if $\{a_j\}_{j\in I}$ is a $\ZZ$-diagonalization of $L$, then 
	\begin{equation}\label{eq:sum-L}
		L=\sum_{j\in I} \Lambda_{a_j}P_{V_{a_j}},
	\end{equation}
	where $P_{V_{a_j}}$ is the orthogonal projection of $V$ onto $V_{a_j}$ and, if $I$ is an infinite set, the convergence is in the strong operator topology sense.


	The next theorem enumerates some results regarding $\ZZ$-diagonalizable operators. Statements \eqref{thm:L-diag-O-diag1} and \eqref{thm:L-diag-O-diag3} are extended versions of \cite[Theorem 6.4]{ACCP} and \cite[Theorem 6.16]{ACCP} respectively.

	\begin{theo}\label{thm:L-diag-O-diag}
		Let $V$ be a $\ZZ$-invariant space with range function $\J$ and $L:V\to V$ a bounded $\ZZ$-preserving operator with range operator $O$. Then, the following statements hold:
		\begin{enumerate}
			\item\label{thm:L-diag-O-diag1} If $L$ is $\ZZ$-diagonalizable, $O(\w)$ is diagonalizable for a.e. $\w\in\Omega$. Moreover, if $\{a_j\}_{j\in I}$ is a $\ZZ$-diagonalization of $L$, then $\sigma_p(O(\w))\subset\{\wh{a}_j(\w)\,:\,j\in I\}$ for a.e. $\w\in\Omega$.
			\item\label{thm:L-diag-O-diag2} If $L$ is $\ZZ$-diagonalizable, $L$ is normal.
			\item\label{thm:L-diag-O-diag3} If $\dim \J(\w)<\infty$ for a.e. $\w\in\Omega$ and $L$ is normal, then $L$ is $\ZZ$-diagonalizable.
		\end{enumerate}
	\end{theo}

	\begin{proof}
		The statement in \eqref{thm:L-diag-O-diag1} follows  straightforwardly from the Definition \ref{def:s-diag} and Proposition \ref{prop:eigen}.
		
		In order to see \eqref{thm:L-diag-O-diag2}, observe that if $L$ is $\ZZ$-diagonalizable, then $O(\w)$ is normal for a.e. $\w\in\Omega$ since it is diagonalizable for a.e. $\w\in\Omega$. Thus, by Theorem \ref{thm:pointwise}, $L$ is normal.
		
		Finally, we prove \eqref{thm:L-diag-O-diag3}. If $L$ is normal then $O(\w)$ is normal for a.e. $\w\in\Omega$ due to Theorem \ref{thm:pointwise}. Thus, we have that $O(\w)$ is a normal operator acting on a finite-dimensional space $\J(\w)$, and hence diagonalizable for a.e. $\w\in\Omega$. 
		
		By Theorem \ref{thm:measurable-eigen}, there exist functions $\lambda_j\in L^\infty(\Omega)$, $j\in\N$, such that for a.e. $\w\in\Omega$ we have the following orthogonal decomposition
		\begin{equation}\label{eq:J-decomp}
		\J(\w)=\underset{j\in\N}{\bigpoplus}\ker(O(\w)-\lambda_j(\w)\mathcal I_{\w}).
		\end{equation}
		We discard the functions such that $\ker(O(\w)-\lambda_j(\w)\mathcal I_{\w})=\{0\}$ for a.e. $\w\in\Omega$. Since for every $j$, $\lambda_j\in L^\infty(\Omega)$, there exists a sequence of bounded spectrum $a_j\in\ell_2(\ZZ)$  such that $\wh{a}_j=\lambda_j$. Then, $\Lambda_{a_j}$ is a $\ZZ$-eigenvalue of $L$ for every $j$ and by \eqref{eq:J-decomp} we have the orthogonal decomposition  $$V=\underset{j}{\bigpoplus} V_{a_j}.$$
	\end{proof}

	Notice that because of Remark \ref{rem:esssup-eigen}, if $V$ is finitely generated and $L$ is $\ZZ$-diagonalizable, then there exists a $\ZZ$-diagonalization where the sum in \eqref{eq:sum-L} is finite.	
	
	In general it is uncertain if the diagonalization of $O(\w)$ for a.e. $\w\in\Omega$ implies that $L$ is $\ZZ$-diagonalizable. This is due to the fact that it depends on the possibility to obtain a measurable selection of the eigenvalues of the range operator.

	\subsection{$\ZZ$-diagonalization on general $\ZZ$-invariant spaces} 	
	
	In this subsection we establish conditions on the range operator of a $\ZZ$-preserving operator acting on a general $\ZZ$-invariant space in order to admit a $\ZZ$-diagonalization. 	
	For this, we will make use of the following lemma.
	
	\begin{lem}\label{lem:castaign-no-multiplicity}
		Let $(X,\mathcal M)$ be a measurable space, and let  $F:X\rightsquigarrow \C$ be a measurable set-valued map to non-empty closed values such that $F(x)\subseteq K$ for every $x\in X$ and $K\subset \C$ a compact set. Then, there exists a sequence of measurable bounded functions $g_j:X\to \C$, $j\in \N$ such that for every  $j\neq j',\;\, g_j(x)\neq g_{j'}(x)$ for every $x\in X$ and 
		\begin{equation}
		F(x)\subset\overline{\{g_j(x)\,:\,j\in\N\}}, \qquad x\in X.
		\end{equation}
		
	\end{lem}
	
	\begin{proof}
		By Theorem \ref{thm:castaign} we have a sequence of measurable functions $f_j:X\to Y$, $j\in\N$ such that 
		$F(x)=\overline{\{f_j(x)\,:\,j\in\N\}}$ for every $x\in X$. We  construct the functions $g_j$, $j\in\N$ inductively.
		Choose $z_0 \notin K$ such that $z_0+j\notin K$ for every $j\in\N$. 
		
		Let $g_1:=f_1$. Now, consider the set $E_2:=\{x\in X\,:\,f_2(x)=g_1(x)\}$. Since both $f_2$ and $g_1$ are measurable functions, we have that $E_2$ is measurable. Now, define $g_2:X\to \C$ as follows,
		\begin{equation}
		g_2(x):=\begin{cases}
		f_2(x)\quad &x\notin E_2\\
		z_0 + 2 & \text{c.c.}
		\end{cases}
		\end{equation}
		It is clear that $g_2$ is a measurable bounded function and $g_1(x)\neq g_2(x)$ for every $x\in X$.
		
		Now, let $E_3:=\{x\in X\,:\,f_3(x)=g_2(x)\}\cup\{x\in X\,:\,f_3(x)=g_1(x)\}$. Again, since $f_3,g_2$ and $g_1$ are measurable functions, $E_3$ is a measurable set. Hence, we define $g_3:X\to \C$ as
		\begin{equation}
		g_3(x):=\begin{cases}
		f_3(x)\quad &x\notin E_3\\
		z_0+3 & \text{c.c.}
		\end{cases}
		\end{equation}
		Again, $g_3$ is a measurable bounded function and $g_3(x)\neq g_2(x) \neq g_1(x)$ for every $x\in X$.
		
		Proceeding this way, in countable steps we obtain a sequence of measurable bounded functions $g_j:X\to \C$, $j\in\N$ such that $g_j(x)\neq g_{j'}(x)$ for $j\neq j'$ and for every $x\in X$. Moreover, it is clear that, by construction, $\{f_j(x)\,:\,j\in\N\}\subset \{g_j(x)\,:\,j\in\N\}$ for every $x\in X$ and so $F(x)\subset \overline{\{g_j(x)\,:\,j\in\N\}}$ for every $x\in X$.
	\end{proof}

	Now, we are ready to state and prove the following theorem.

	\begin{theo}\label{thm:diagonalizable-fibers-s-diag}
		Let $V\subseteq L^2(\RR)$ be a $\ZZ$-invariant space with range function $\J$. Let $L:V\to V$ be a bounded, normal and $\ZZ$-preserving operator with range operator $O:\J\to\J$. Suppose that $O(\w)$ is diagonalizable and all its eigenvalues are isolated points of $\sigma(O(\w))$ for a.e. $\w\in\Omega$. Then, $L$ is $\ZZ$-diagonalizable.
	\end{theo}
	
	\begin{proof}
		
			By Theorem \ref{thm:measurable-spectrum} and Lemma \ref{lem:castaign-no-multiplicity}, there exists a sequence of measurable and bounded functions $g_j:\Omega\to\C$, $j\in \N$ such that  $g_j(\w)\neq g_{j}'(\w)$ for $j\neq j'$ and $\sigma(O(\w))\subseteq \overline{\{g_j(\w)\,:\,j\in \N\}}$ for a.e. $\w\in\Omega$. Furthermore, since all the eigenvalues of $O(\w)$ are isolated points of $\sigma(O(\w))$ for a.e. $\w\in\Omega$, then
			\begin{equation}
			\sigma_p(O(\w))\subset\{g_j(\w)\,:\,j\in \N\}
			\end{equation}
			for a.e. $\w\in\Omega$. 
			
			Since $O(\w)$ is diagonalizable for a.e. $\w\in\Omega$, the following equality holds
			\begin{equation}\label{eq:J-decomp-infinite dim}
			\J(\w)=\underset{j\in\N}{\bigpoplus} \ker\left(O(\w)-g_j(\w)\mathcal I_{\w}\right)
			\end{equation}
			where the sum is orthogonal. Notice that $\ker\left(O(\w)-g_j(\w)\mathcal I_{\w}\right)$ could be $\{0\}$ for some $j$ and some set of  positive measure. However, we discard all the functions $g_j$ such that $\ker\left(O(\w)-g_j(\w)\mathcal I_{\w}\right)=\{0\}$ almost everywhere.
			
			Now, since $g_j$ is measurable and bounded, there exists a sequence of bounded spectrum $a_j\in\ell_2(\ZZ)$ such that $\wh{a}_j = g_j$. Then, $\Lambda_{a_j}$ is a $\ZZ$-eigenvalue of $L$ for all $j$ and by \eqref{eq:J-decomp-infinite dim} we get the orthogonal decomposition $$V=\underset{j}{\bigpoplus} V_{a_j}.$$

	\end{proof}

In what follows we discuss two examples of operators satisfying that $O(\w)$ is diagonalizable and $\sigma_p(O(\w))$ are all isolated points of $\sigma(O(\w))$ for a.e. $\w\in \Omega$ and hence $\ZZ$-diagonalizable. As a first example, we give the following case.

	\begin{exa}\label{ex:compact}
	 Let $L:V\to V$ be a bounded, normal, injective and $\ZZ$-preserving operator such that $O(\w)$ is compact for a.e. $\w\in\Omega$.  For this case, by Theorem \ref{thm:pointwise}, $O(\w)$ is normal and injective for a.e. $\w\in\Omega$. Hence, $O(\w)$ is diagonalizable and its eigenvalues are all isolated points of $\sigma(O(\w))$ for a.e. $\w\in\Omega$, and thus, by Theorem \ref{thm:diagonalizable-fibers-s-diag}, $L$ is $\ZZ$-diagonalizable.
	\end{exa}	

	Notice that $O(\w)$ being compact a.e. $\w\in\Omega$ does not imply that $L$ is compact, this can be seen as a consequence of the following two results. 
	
	\begin{prop}\label{thm:spectrum-finite-mult}
		Let $L:V\rightarrow V$ be a bounded $\ZZ$-preserving operator. Then, $L$ does not have eigenvalues with finite multiplicity. 
	\end{prop}
	\begin{proof}
		Suppose there is an eigenvalue $\lambda$ of finite multiplicity. Then, $E_\lambda = \ker(L-\lambda\mathcal I)\neq \{0\}$. Since $E_\lambda$ is a $\ZZ$-invariant space and  the only finite dimensional $\ZZ$-invariant space is the zero space, we obtain a contradiction.
	\end{proof}

	
	\begin{cor}\label{coro:compact-sp}
		Let $L:V\to V$ be a bounded $\ZZ$-preserving operator. Then, $L$ is compact if and only if $L=0$.
	\end{cor}

	\begin{proof}
		If $L$ is compact,  $L^*L$ is bounded, normal, compact and $\ZZ$-preserving . Hence, $L^*L$ is diagonalizable. Moreover, by compactness, every eigenvalue $\lambda\neq 0$ of $L^*L$ should be of finite multiplicity. By Proposition \ref{thm:spectrum-finite-mult}, we deduce that $\lambda=0$ is the only possible eigenvalue. Hence $L^*L=0$ and so $L=0$.
	\end{proof}
	
	For instance, let $V$ be a finitely generated $\ZZ$-invariant space with range function $\J$ and take any bounded $\ZZ$-preserving operator $L\neq 0$ acting on $V$.  Then, $O(\w):\J(\w)\rightarrow \J(\w)$ is compact since $\dim \J(\w)<\infty$ for a.e. $\w\in\Omega$ but $L$ is not.


	\begin{rem}
		Let $L$ be a bounded, normal and $\ZZ$-preserving operator such that $O(\w)$ is compact for a.e. $\w\in\Omega$ but not necessarily injective. Given $V':= \ker(L)^\perp$ the orthogonal complement of $\ker(L)$ in $V$, we have that $V'$ is a $\ZZ$-invariant space which is also invariant by $L$. Thus, define $L':=L|_{V'}:V'\to V'$, then $L'$ is an operator satisfying the same properties as in Example \ref{ex:compact}. Hence, $L'$ is $\ZZ$-diagonalizable on $V'$. Given a $\ZZ$-diagonalization $\{a_j\}_{j\in I}$ of $L':V'\to V'$, we can decompose $V$ as follows
		$$V=\ker(L)\poplus \underset{j}{\bigpoplus} V'_{a_j}.$$
		where $V'_{a_j}$ are $\ZZ$-eigenspaces of $L'$.
	\end{rem}
	
	We now give some sufficient conditions for a $\ZZ$-preserving operator to admit compact range operator a.e. $\w\in\Omega$.
	
	\begin{prop}\label{prop:O(w)-compact}
		Let $V$ be a $\ZZ$-invariant space with range function $\J$ and $L:V\to V$ a bounded $\ZZ$-preserving operator with range operator $O$. 
		\begin{enumerate}
			\item If $V'=\overline{L(V)}$ is a $\ZZ$-invariant space satisfying that $\dim \J_{V'}(\w)<\infty$ for a.e. $\w\in\Omega$, then $O(\w)$ is of finite rank a.e. $\w\in\Omega$. We will call these operators of finite range rank.
			\item\label{prop:O(w)-compact-2} If there exists a sequence $\{L_n\}_{n\in\N}$ such that $L_n:V\to V$ is a bounded $\ZZ$-preserving operator of finite range rank for every $n\in\N$, and $L_n\to L$ when $n\to \infty$ uniformly, then $O(\w)$ is compact for a.e. $\w\in\Omega$.
		\end{enumerate}
	\end{prop}

	\begin{proof}
		$(1)$ Recall that by item \eqref{rank} in Theorem \ref{thm:pointwise} the range function associated to  $\overline{L(V)}$ is the one given by $\overline{O(\w)\J(\w)}$ for a.e. $\w\in\Omega$. Thus, $O(\w)$ is a finite rank operator for a.e. $\w\in\Omega$.
		
		$(2)$ Let $O_n$ be the range operator associated to each $L_n$ for every $n\in \N$, then by \eqref{op-norm} in Theorem \ref{thm:pointwise} we have that $O_n(\w)\to O(\w)$ when $n\to \infty$ for a.e. $\w\in\Omega$. However using  $(1)$ of this proposition we have that $O_n(\w)$ is of finite rank for every $n\in\N$, thus $O(\w)$ is compact for a.e. $\w\in\Omega$.
	\end{proof}
	
	For the second example we need to give the following definition.
	
	\begin{defi}
		Let $\Hil$ be a separable Hilbert space, $A:\Hil\to\Hil$ a normal bounded operator, $I$ a finite index set and $\{f_i\}_{i\in I}\subset\Hil$. We say that $(\Hil, A,\{f_i\}_{i\in I})$ is a DS-triple if $\{A^nf_i\,:\,n\in \N, i\in I\}$ is a frame of $\Hil$. In that case, we say that $A$ admits a DS-triple.
	\end{defi}

	The problem of finding conditions on $\Hil$, $A$ and $\{f_i\}_{i\in I}$ under which $(\Hil, A, \{f_i\}_{i\in I})$ is a DS-triple has been well studied and is of special interest in the context of dynamical sampling theory, see \cite{ADK13,ADK15,ACCMP2017,ACUS2017,CMPP2019}. The following result has been proved in \cite{CMPP2019}.
	
	\begin{theo}[\cite{CMPP2019}]\label{thm:DS-diagonalizable}
		Let $\Hil$ be an infinite-dimensional separable Hilbert space,  $A:\Hil\to\Hil$ a bounded normal operator and $\{f_i\}_{i\in I}\subset\Hil$ with $I$ a finite index set. If $(\Hil, A, \{f_i\}_{i\in I})$ is a DS-triple, then $A$ is diagonalizable and $\sigma_p(A)\subset \mathbb D$, where $\mathbb D=\{\lambda\in\C\,:\, |\lambda|<1\}$. Moreover, in this case, the dimension of each eigenspace is less than or equal to $\#I$ and the cluster points of $\sigma_p(A)$ are contained in $S_1=\{\lambda\in\C\,:\, |\lambda|=1\}$.
	\end{theo}

	\begin{rem}
		In a similar way as in Corollary \ref{coro:compact-sp}, one can also prove that if a $\ZZ$-preserving operator $L:V\to V$ admits a DS-triple $(V,L,\{f_i\}_{i\in I})$ with $I$ a finite index set, then $L=0$. However, a $\ZZ$-preserving operator $L\neq 0$ could satisfy that $O(\w)$ admits a DS-triple for a.e. $\w\in\Omega$.
	\end{rem}	 

	 Observe that a normal compact operator acting on an infinite-dimensional Hilbert space never admits a DS-triple since the only cluster point of its eigenvalues, if any, is zero. Thus, we have the following example.
	 
	 \begin{exa}
	 	Let $V$ be a $\ZZ$-invariant space with range function $\J$ such that $\dim \J(\w)=\infty$ for a.e. $\w\in\Sigma(V)$. Let $L:V\to V$ be a bounded, normal $\ZZ$-preserving operator such that $O(\w)$ is admits a DS-triple for a.e. $\w\in\Sigma(V)$. By Theorem \ref{thm:pointwise}, $O(\w)$ is normal for a.e. $\w\in\Omega$. Then, by Theorem \ref{thm:DS-diagonalizable}, $O(\w)$ is diagonalizable and its eigenvalues are all isolated points of $\sigma(O(\w))$ for a.e. $\w\in\Omega$. Then, by Theorem \ref{thm:diagonalizable-fibers-s-diag}, $L$ is $\ZZ$-diagonalizable.
	 \end{exa}

	 Finally we give a sufficient condition for a $\ZZ$-preserving operator in order to guarantee that its fibers admit a DS-triple which is an immediate consequence of Theorem \ref{thm:frame-fiber}.
	 
	 \begin{prop}
	 		Let $V$ be a $\ZZ$-invariant space with range function $\J$. Let $L:V\to V$ a bounded $\ZZ$-preserving operator with range operator $O$. Assume that there exist functions $\{f_i\}_{i\in I}$, with $I$ a finite index set, such that $\{\tr{k}L^jf_i \, : \, k \in \ZZ, j\in\N, i\in I\}$ is a frame of $V$, then $\left(\J(\w),O(\w),\{\mathcal T[f_i](\w)\}_{i\in I}\right)$ is a DS-triple for a.e. $\w\in\Omega$.
	 \end{prop}
 
 	\section{$\Gamma$-preserving operators and $\Gamma$-diagonalization}\label{sec:Gamma-inv}

    Throughout this section, we will consider a discrete and at most countable group $G$ acting on $\RR$ by the continuous automorphisms
	$x \mapsto g x \in \RR$ for $g \in G$ and $x \in \RR$.
	As before, $\Delta$ is a uniform lattice of $\RR$ and we will assume that the action of $G$ on $\RR$ preserves $\ZZ$, that is $g \ZZ = \ZZ$ for all $g \in G$. This implies in particular that the action of $G$ preserves the Haar measure of $\RR$, i.e.
	$$
	|gE| = |E| \ , \quad \forall \ E \subset \RR \textnormal{ measurable}, \quad \forall \ g \in G .
	$$
	This can be proved as follows. Since $G$ acts on $\RR$ by automorphisms, then (see e.g. \cite[Theorem 15.26]{HR1979})
	there exists a homomorphism $\delta : G \to \R^+$ such that, for all measurable $E \subset \RR$ we have $|gE| = \delta(g)|E|$. Let $Q \subset \RR$ be a fundamental domain for $\RR/\ZZ$. Its measure $|Q|$ is finite and, since the action of $G$ preserves $\Delta$, then for all $g \in G$ the set $gQ$ is a fundamental domain, so $|gQ| = |Q|$. Thus, $\delta = 1$.
	
	The action of $G$ on $\RR$ induces an action of $G$ on $\wh{\RR}$ by duality:
	$$
	\langle g^* \xi , x \rangle := \langle \xi, g x\rangle, \quad g\in G,\,\xi\in\wh{\RR},\,x\in \RR.
	$$
	This dual action of $G$ on $\wh{\RR}$ satisfies $g_1^* g_2^* = (g_2 g_1)^*$ for all $g_1, g_2 \in G$, and it preserves $\ZZ^\bot$ and the Haar measure of $\wh{\RR}$. Moreover, it induces an action on the quotient group $\wh{\RR}/\ZZ^\bot$ by
	$$
	g^*[\xi] := [g^* \xi] \ , \quad \xi \in \wh{\RR}, \quad g \in G,
	$$
	where $[\xi]$ is the class of $\xi$ in $\wh{\RR}/\ZZ^\bot$. This implies that we can define an action of $G$ on any Borel section $\Omega \subset \wh{\RR}$ of $\wh{\RR}/\ZZ^\bot$, as follows. Let $q_\Omega : \wh{\RR} \to \Omega$ be the canonical section, that is, $q_\Omega(\xi)$ is the unique point in $[\xi] \cap \Omega$,
	and denote by $\nu_\Omega : \wh{\RR} \to \ZZ^\perp$ the map 
	\begin{equation}\label{eq:sigmaOmega}
	\nu_\Omega(\xi) = q_\Omega(\xi) - \xi \ , \quad \xi \in \wh{\RR}.
	\end{equation}
	Then, the maps $\{g^\sharp : \Omega \to \Omega, g \in G\}$ given by
	\begin{equation}\label{eq:gsharp}
	g^\sharp \omega = q_\Omega(g^*\omega)
	\end{equation}
	define an action, satisfying $g_1^\sharp g_2^\sharp = (g_2 g_1)^\sharp$. Indeed, we have
	\begin{align*}
	g_1^\sharp g_2^\sharp \omega & = q_\Omega(g_1^* g_2^\sharp \omega) = q_\Omega(g_1^* q_\Omega(g_2^* \omega) ) = q_\Omega(g_1^* (g_2^* \omega + \nu_\Omega(g_2^* \omega) ) )\\
	& = q_\Omega(g_1^* g_2^* \omega + g_1^*\nu_\Omega(g_2^* \omega) ) = q_\Omega(g_1^* g_2^* \omega ) = q_\Omega( (g_2 g_1)^* \omega )
	\end{align*}
	where the second to last identity is due to the fact that $g_1^*\nu_\Omega(g_2^* \omega) \in \ZZ^\perp$.
	The action (\ref{eq:gsharp}) will coincide with the dual action of $G$ on $\wh{\RR}$ only when $\Omega$ is an invariant subset of $\wh{\RR}$ for the dual action.

	Given that the action of $G$ preserves $\ZZ$, we can define the semidirect product $$\Gamma = \ZZ \rtimes G = \{(k,g) \, : \, k \in \ZZ, g \in G\},$$ with composition law
	$$
	(k,g) \cdot (k',g') = (k + gk', gg').
	$$
	The action of $\Gamma$ on $\RR$ reads
	$$
	\gamma x = gx + k \ , \quad \gamma = (k,g) \in \Gamma \, , \ x \in \RR.
	$$
		
	A motivational example for this setting is given by the crystal (or crystallographic) groups. 
	
	\begin{defi}
		A crystal group $\Gamma$ is a discrete subgroup of the isometries of $\R^d$ that has a closed and bounded Borel section $P$, that is,
		\begin{enumerate}
			\item $\bigcup_{\gamma\in\Gamma} \gamma P = \R^d$.
			\item If $\gamma \neq \gamma'$, then $|\gamma P \cap \gamma'P|=0$.
		\end{enumerate}
	\end{defi}
	
	There is a subclass of the crystal groups we are interested in.
	
	\begin{defi}\label{def:crystal-split}
		We say that a crystal group $\Gamma$ splits if it is the semidirect product $\Gamma=\ZZ\rtimes G$ of a finite group $G$ and a uniform lattice $\ZZ$ of $\R^d$.
	\end{defi}
	
	In particular, it can be seen that any crystal group can be embedded in a crystal group that splits. We refer the reader to \cite{Bie1910,Far1981,GS1987} for more general results on these groups.

	We will now define some unitary representations which play a fundamental role in what follows. 
	\begin{defi}
	We denote by 
	$R : G \to \U(L^2(\RR))$ the unitary representation defined by
	$$
	\rot{g}f(x) = f(g^{-1}x) \ , \quad f \in L^2(\RR),\, g\in G.
	$$
	\end{defi}
	Note that, since $\rot{g}\tr{k} = \tr{gk}\rot{g}$, the map $(k,g) \mapsto \tr{k}\rot{g}$ defines a unitary representation of the semidirect product group $\Gamma = \ZZ \rtimes G$ on $L^2(\RR)$. Also, for all $f \in L^2(\RR)$ 
	and all $g \in G$, the following relation holds:
	\begin{equation}\label{relations-tau}
	\wh{\rot{g} f}(\xi) = \wh{f}(g^* \xi).
	\end{equation}

	\begin{defi}\label{def:rpeque}
		We denote by $t : \ZZdual \to \U(\ell_2(\ZZdual))$ the left regular representation of $\ZZdual$, that is
		$$
		t_\ell a(\ell') = a (\ell' - \ell) \ , \quad a \in \ell_2(\ZZdual) ,\, \ell, \ell' \in \ZZdual.
		$$
		and by $r : G \to \U(\ell_2(\ZZdual))$ the unitary representation defined by
		$$
		r_g a (\ell) = a(g^* \ell) \ , \quad g\in G, \,a \in \ell_2(\ZZdual),\, \ell \in \ZZdual.
		$$
	\end{defi}
    As for the previous case, since $r_g t_\ell = t_{(g^*)^{-1}\ell}r_g$, and recalling that $g_1^* g_2^* = (g_2 g_1)^*$, the map $(\ell,g) \mapsto t_\ell r_g$ defines a unitary representation of the semidirect product group $\ZZdual \rtimes G$ on $\ell_2(\ZZdual)$.

    Finally, we introduce the following representation of $G$ on $L^2(\Omega,\ell_2(\ZZdual))$.
	\begin{defi}\label{def:Pi}
		We denote by $\Pi$ the unitary representation of $G$ on the Hilbert space $L^2(\Omega,\ell_2(\ZZdual))$ defined by
		$$\Pi(g) = \T R_g \T^{-1}.$$
	\end{defi}
    
    The explicit form of the representation $\Pi$ is provided by the following proposition.
    \begin{prop}\label{prop:explicitPI}
    Let $\Omega$ be a fundamental set for $\wh{\RR}/\ZZdual$ and let $\pi : \Omega \times G \to \U(\ell_2(\ZZdual))$ be the map
    $$
    \pi^\omega(g) a (\ell) = r_g t_{\nu_\Omega(g^*\omega)} a(\ell) = a(g^*\ell - \nu_\Omega(g^*\omega))
    $$
    where $\nu_\Omega$ is given by (\ref{eq:sigmaOmega}). 
    Then the representation $\Pi$ reads explicitly
    \begin{equation}\label{eq:explicitPi}
    \Pi(g)F(\omega) = \pi^\omega(g) F(g^\sharp \omega) , \quad F \in L^2(\Omega,\ell_2(\ZZdual)),\, g \in G,\, \textnormal{a.e. } \omega \in \Omega.
    \end{equation}
    Moreover, the map $(\omega,g) \mapsto \pi^\omega(g)$ satisfies $\pi^\omega(\id) = \mathbb{I}_{\ell_2(\ZZdual)}$ and
    \begin{equation}\label{eq:covariance}
    \pi^\omega(g_1g_2) = \pi^\omega(g_1)\pi^{g_1^\sharp\omega}(g_2).
    \end{equation}
    In particular, if $\Omega$ is invariant under the dual action of $G$ on $\wh{\RR}$, then $\pi^\omega(g) = r_g$.
    \end{prop}
    \begin{proof}
    By (\ref{eq:sigmaOmega}), (\ref{eq:gsharp}) and (\ref{relations-tau}), for all $f \in L^2(\RR)$ we have that
    \begin{align*}
    \T[\rot{g}f](\omega) & = \{\wh{f}(g^*\omega + g^*\ell)\}_{\ell \in \ZZdual} = \{\wh{f}(g^\sharp\omega + g^*\ell - \nu_\Omega(g^*\omega) )\}_{\ell \in \ZZdual}\\
    & = \pi^\omega(g) \T[f](g^\sharp \omega)
    \end{align*}
    which proves (\ref{eq:explicitPi}) using that $\T$ is an isomorphism. In order to prove (\ref{eq:covariance}), recall that $\Pi$ is a unitary representation, because it is defined as the intertwining of a unitary representation with an isomorphism of Hilbert spaces. Using (\ref{eq:explicitPi}), this implies that
    $$
    \Pi(g_1)\Pi(g_2)F(\omega) = \Pi(g_1 g_2) F(\omega) = \pi^\omega(g_1g_2)F(g_2^\sharp g_1^\sharp \omega).
    $$
    On the other hand, we have
    $$
    \Pi(g_1)\Pi(g_2)F(\omega) = \pi^\omega(g_1)\Pi(g_2)F(g_1^\sharp\omega) = \pi^\omega(g_1)\pi^{g_1^\sharp\omega}(g_2)F(g_2^\sharp g_1^\sharp \omega).
    $$
    Since both relations hold for all $F \in L^2(\Omega,\ell_2(\ZZdual))$, all $g \in G$ and a.e.  $\omega \in \Omega$, this proves (\ref{eq:covariance}).
    \end{proof}

%
%
%

	\subsection{$\Gamma$-invariant spaces}
	
	Given this setting, we are now interested in the subspaces of $L^2(\RR)$ that are invariant under the action of the unitary representation $T_kR_g$. These spaces have been first studied in great detail in \cite{BCHM}.
	
	\begin{defi}
		We say that a closed subspace $V \subset L^2(\RR)$ is $\Gamma$-invariant if $\tr{k} \rot{g} V \subset V$ for all $(k,g) \in \Gamma.$ 
	\end{defi}
	
	A $\Gamma$-invariant space is, in particular, $\ZZ$-invariant as it can be seen that $V$ is $\Gamma$-invariant if
	$$
	f \in V \ \Rightarrow \ \tr{k}f \in V \ \ \forall \ k \in \ZZ \, , \ \textnormal{and} \ \rot{g}f \in V \ \ \forall \ g \in G.
	$$
	
	Consequently, $V$ is $\Gamma$-invariant if and only if $V$ is $\ZZ$-invariant, and $\Pi(g)\T[V] \subset \T[V]$ for every $g \in G$, where $\Pi$ is the representation given in Definition \ref{def:Pi}.
	
	The following theorem gives a characterization of $\Gamma$-invariant closed subspaces in terms of a covariance property of the range function associated to its $\ZZ$-invariant subspace. 
	This theorem generalizes the one in \cite[Theorem 3.3]{BCHM} where a proof was given under the additional hypothesis that the fundamental domain $\Omega$ is invariant under the dual action of $G$ on 
	$\wh\RR$. However, in that paper this additional hypothesis was omitted, as one of the reviewers of this paper pointed out.
	
	\begin{theo}\label{theo:rangecovariance}
		A closed subspace $V$ of $L^2(\RR)$ is $\Gamma$-invariant if and only if it is $\ZZ$-invariant and its associated range function $\J$ satisfies
		\begin{equation}\label{eq:rangecovariance}
		\J(\omega) = \pi^\omega(g) \J(g^\sharp\omega) \, , \ \textnormal{a.e.} \ \omega \in \Omega \, , \ \forall g \in G.
		\end{equation}
	\end{theo}

	\begin{proof}
	Assume that $V$ is $\ZZ$-invariant with range function $\J$ and, for $g \in G$, let us denote by $V_g = \rot{g}(V)$. Then $V_g$ is also $\ZZ$-invariant, and we denote by $\J_g$ its range function. We claim that
	$$
	\J_g(\omega) = \pi^\omega(g) \J(g^\sharp\omega) \, , \ \textnormal{a.e.} \ \omega \in \Omega \, , \ \forall g \in G.
	$$
	Note that, by proving this claim, we also prove the statement of the Theorem.
	
	To prove the claim recall that, by Theorem \ref{Th:Helson}, $F \in \T[V] \iff F(\omega) \in \J(\omega)$ for a.e. $\omega \in \Omega$ and, by (\ref{eq:explicitPi}), for $F \in \T[V]$ we have $\Pi(g)F(\omega) = \pi^\omega(g)F(g^\sharp\omega)$.
	
	Assume first that, for a given $g \in G$, we have a function $H \in L^2(\Omega,\ell_2(\ZZdual))$ such that $H(\omega) \in \J_g(\omega)$ for a.e. $\omega \in \Omega$, that is $H \in \T[\rot{g}(V)]$. We want to prove that $H(\omega) \in \pi^\omega(g) \J(g^\sharp\omega)$ for a.e. $\omega \in \Omega$. By Definition \ref{def:Pi}, we have $\Pi(g^{-1})H \in \T[V]$, which implies
	$$
	\pi^\omega(g^{-1}) H((g^{-1})^\sharp\omega) \in \J(\omega) \, , \ \textnormal{a.e.} \ \omega \in \Omega.
	$$
	Denoting by $\omega' = (g^{-1})^\sharp\omega$, this reads equivalently
	$$
	\pi^{g^\sharp\omega'}(g^{-1}) H(\omega') \in \J(g^\sharp\omega') \, , \ \textnormal{a.e.} \ \omega' \in \Omega.
	$$
	Thus, by applying $\pi^{\omega'}(g)$ on both sides, and using Proposition \ref{prop:explicitPI}, we obtain $H(\omega) \in \pi^\omega(g) \J(g^\sharp\omega)$ for a.e. $\omega \in \Omega$.
	
	Assume now that, for a $g \in G$, we have an $H \in L^2(\Omega,\ell_2(\ZZdual))$ such that $H(\omega) \in \pi^\omega(g) \J(g^\sharp\omega)$ for a.e. $\omega \in \Omega$, and let $F = \Pi(g^{-1})H$. Then
	$$
	H(\omega) = \pi^\omega(g) F(g^\sharp\omega).
	$$
	Since $\pi^\omega(g)$ is a unitary operator on $\ell_2(\ZZdual)$, and $g^\sharp$ is a bijection of $\Omega$ we have obtained that $F(\omega) \in \J(\omega)$ for a.e. $\omega \in \Omega$. That is, $F \in \T[V]$, or, equivalently, $H \in \Pi(g)(\T[V])=\T[\rot{g}V]$. Thus, $H(\omega) \in \J_g(\omega)$ for a.e. $\omega \in \Omega$.
	\end{proof}

	\subsection{$\Gamma$-preserving operators}\label{sec:Gamma-pres}
	
	In this subsection we consider operators defined on $\Gamma$-invariant spaces which commute with the unitary representation $T_kR_g$. 
	
	\begin{defi}
		Let $V, V'\subset L^2({\RR})$ be two $\Gamma$-invariant spaces. We say that a bounded operator ${L:V\to V'}$ is $\Gamma$-preserving if $LT_kR_g = T_kR_g L$ for every $k\in\ZZ$ and $g\in G$.	
	\end{defi}
	
	Observe that, in particular, $L$ is $\Gamma$-preserving if and only if $L$ is $\ZZ$-preserving and $LR_g=R_gL$ for every $g\in G$. We will focus on bounded $\Gamma$-preserving operators acting on a $\Gamma$-invariant $V$, that is $L:V\to V$. Since $L$ is $\ZZ$-preserving, there exists a corresponding range operator $O:\J\to \J$. 
	
	In the same spirit of Theorem \ref{theo:rangecovariance}, we have the following result.

	\begin{theo}\label{O(g*w)}
		Let $V\subset L^2({\RR})$ be a $\Gamma$-invariant space and $L:V\to V$ a bounded operator with corresponding range operator $O$. Then, $L$ is $\Gamma$-preserving if and only if it is $\ZZ$-preserving and for all $g\in G$ and a.e. $\w\in\Omega$,
		\begin{equation}\label{eq:O(g*w)}
		O(g^\sharp\w)=\pi^{g^\sharp\w}(g^{-1})O(\w)\pi^\omega(g).
		\end{equation}
	\end{theo}
	\begin{proof}
		Assume that $L$ is $\Gamma$-preserving. Fix $g\in G$ and note that for every $f\in V$, and for a.e. $\w\in \Omega$
		\begin{align*}
		O(\w)\left(\Pi(g)\mathcal T[f](\w)\right) &= O(\w)\left(\mathcal T[R_gf](\w)\right) =\mathcal T[LR_gf](\w)\\
		&=\mathcal T[R_gLf](\w) = \Pi(g)\mathcal T[Lf](\w)	\\
		&=\Pi(g)O(\w)\mathcal T[f](\w).
		\end{align*}
		Hence, if $F\in \mathcal T[V]$, by \eqref{eq:explicitPi} we have that for a.e. $\w\in \Omega$
		$$O(\w)\pi^\omega(g)F(g^\sharp\w)=\pi^\omega(g)O(g^\sharp\w)F(g^\sharp\w).$$
		Since, by (\ref{eq:covariance}), $(\pi^\omega(g))^{-1} = \pi^{g^\sharp\w}(g^{-1})$, we deduce that for a.e. $\w\in\Omega$,
		\begin{equation*}
		O(g^\sharp\w)=\pi^{g^\sharp\w}(g^{-1})O(\w)\pi^\w(g).
		\end{equation*}

		For the converse, if \eqref{eq:O(g*w)} holds, then for every $F\in\T[V]$ we have that for a.e. $\w\in\Omega$ and for every $g\in  G$, $$O(\w)\Pi(g)F(\w)=\Pi(g)O(\w)F(\w).$$ By the computation above, this means that for every $f\in V$ and for every $g\in G$, $\T[LR_gf]=\T[R_gLf]$. Thus,  $LR_g=R_gL$ for every $g\in G$.
	\end{proof}
	
	As a consequence, we see that much of the structure of $O$ is preserved by the action of $G$ on $\Omega$. In particular, we have the next proposition concerning the spectra of $O(\w).$
	
	\begin{prop}\label{prop:spectrum-invariant}
		Let $V\subset L^2({\RR})$ be a $\Gamma$-invariant space and $L:V\to V$ a bounded $\Gamma$-preserving operator with corresponding range operator $O$. Then for all $g\in G$ and a.e. $\w\in\Omega$,
		\begin{enumerate}
			\item $\sigma(O(\w))=\sigma(O(g^\sharp\w))$.
			\item\label{prop:spectrum-invariant-2} $\sigma_p(O(\w))=\sigma_p(O(g^\sharp\w))$.
		\end{enumerate}
	\end{prop}
	
	\begin{proof} 
		Fix $g\in G$ and $\w\in \Omega$ where $\J$ and $O$ are defined.
		Assume that $\lambda\in\sigma(O(\w))$, then $O(\w)-\lambda \mathcal I_{\w}$ is not invertible in $\J(\w)$. Thus, \begin{align*}
		O(g^\sharp\w)-\lambda \mathcal I_{g^\sharp\w} & = \pi^{g^\sharp\w}(g^{-1}) O(\w) \pi^{\w}(g) - \lambda\, \pi^{g^\sharp\w}(g^{-1})\pi^{\w}(g)\\
		&= \pi^{g^\sharp\w}(g^{-1}) (O(\w)-\lambda \mathcal I_{\w})\pi^{\w}(g),
		\end{align*}
		which implies that $O(g^\sharp\w)-\lambda \mathcal I_{g^\sharp\w}$ is not invertible in $\J(g^\sharp\w)$, hence proving $(1)$. Now, to prove $(2)$, suppose $\lambda\in\C$ is an eigenvalue of $O(\w)$, then there exists $v\neq 0$ and $v\in\ker\left(O(\w)-\lambda\mathcal I_{\w}\right)$. We will see that $\pi^{g^\sharp\w}(g^{-1})v\in\ker \left(O(g^\sharp\w)-\lambda\mathcal I_{g^\sharp\w}\right)$. Indeed, we have that
		\begin{align*}
		O(g^\sharp\w)(\pi^{g^\sharp\w}(g^{-1})v) & = \pi^{g^\sharp\w}(g^{-1})O(\w)\pi^{\w}(g)(\pi^{g^\sharp\w}(g^{-1}) v)= \pi^{g^\sharp\w}(g^{-1}) O(\w) v\\
		& = \pi^{g^\sharp\w}(g^{-1}) \lambda\, v =\lambda \,(\pi^{g^\sharp\w}(g^{-1}) v).
		\end{align*}
		Since $v\neq 0$, then $\pi^{g^\sharp\w}(g^{-1}) v\neq 0$ and consequently $\ker \left(O(g^\sharp\w)-\lambda\mathcal I_{g^\sharp\w}\right)\neq\{0\}$.
	\end{proof}
	
	We remark that given a measurable function $\lambda:\Omega\to\C$ such that $\lambda(\w)$ is an eigenvalue of $O(\w)$ for a.e. $\w\in\Omega$ the proposition above does not imply that $\lambda(\w)=\lambda(g^\sharp\w)$ for every $g\in G$ and a.e. $\w\in\Omega$.
	
	Since we are interested in obtaining a diagonalization for $\Gamma$-preserving operators similar to Definition \ref{def:s-diag}, we must find some suitable operators to play the role of $\Gamma$-{\it eigenvalue}. The natural choice would be $K_{a}=\sum_{(s,h)\in \Gamma} a(s,h)T_sR_h$ for some sequence  $a=\{a(s,h)\}_{(s,h)\in\Gamma}$ satisfying certain conditions. However, if we require that these operators commute with $T_kR_g$ for every $k\in\ZZ$ and $g\in G$, it is not difficult to see that we are left only with a multiple of the identity.

	Hence, we turn to the $\ZZ$-preserving operators $\Lambda_a$ of Definition \ref{def:Lambda_a}, with $a\in\ell_2(\ZZ)$ of bounded spectrum. We are interested in characterizing such operators that commute with the unitary representation $R_g$ of $G$ on $L^2(\RR)$.
	
	For the next proposition, we introduce the following representation.
	
	\begin{defi}\label{def:tilder}
		We denote by $\tilde{r} : G\to \mathcal U (\ell_2(\ZZ))$ the representation defined by
		$$(\tilde{r}_g(a))(s) = a(g^{-1}s), \quad g\in G,\, a\in\ell_2(\ZZ),\,s\in\ZZ.$$
	\end{defi}
	
	\begin{prop}\label{prop:Lambda_a-G-preserving} 
		Let $a\in\ell_2(\ZZ)$ of bounded spectrum and let $\Lambda_a$ an operator as in Definition \ref{def:Lambda_a}. Then, the following statements are equivalent.
		\begin{enumerate}
			\item $R_g\Lambda_a = \Lambda_a R_g$ for every $g\in G$.
			\item\label{invariance-a-r_g} For every $g\in G$, $\tilde{r}_g(a) = a$.
			\item\label{invariance-a} For every $g\in G$, $\wh{a}(g^{\sharp}\w)=\wh{a}(\w)$ for a.e. $\w\in \Omega$.
		\end{enumerate} 
	\end{prop}
	
	\begin{proof} 
	For a given $g\in G$, first compute for $\w\in \Omega$,
		\begin{align}\label{eq:Fourier-tilde-r_g}
		\wh{(\tilde{r}_g(a))}(\w) &= \sum_{s\in\ZZ} (\tilde{r}_g(a))(s)e^{-2\pi i \w\cdot s}=\sum_{s\in\ZZ} a(g^{-1}s)e^{-2\pi i \w\cdot s}\nonumber\\
		&=\sum_{s\in\ZZ} a(s)e^{-2\pi i \w\cdot gs}	=	\sum_{s\in\ZZ} a(s)e^{-2\pi i g^*\w\cdot s}\nonumber\\
		&=\sum_{s\in\ZZ} a(s)e^{-2\pi i g^\sharp\w\cdot s} = \wh{a}(g^\sharp\w).
		\end{align}
		Now, recalling that $\Pi(g)=\T R_g \T^{-1}$, we see that
		\begin{equation}\label{eq:RgLa}
		R_g\Lambda_a = R_g\T^{-1}M_{\hat{a}}\T = \T^{-1}\Pi(g)M_{\hat{a}}\T.
		\end{equation}
		On the other hand, observe that for $F\in L^2(\Omega, \ell_2(\ZZdual))$ and using \eqref{eq:explicitPi}
		\begin{align*}
			\Pi(g)M_{\hat{a}}F(\w) &= \Pi(g) \,\wh{a}(\w)\,F(\w) = \pi^\w(g)\, \wh{a}(g^\sharp \w)\,F(g^\sharp\w) \\
			&= \wh{a}(g^\sharp \w) \,\pi^\w(g)\,F(g^\sharp\w) = M_{\hat{a}(g^\sharp \cdot)} \,\Pi(g)F(\w).
		\end{align*}
		From the last equality, together with \eqref{eq:Fourier-tilde-r_g} and \eqref{eq:RgLa} it follows that
		\begin{equation*}
			R_g\Lambda_a = \T^{-1}  M_{\hat{a}(g^\sharp \cdot)} \,\Pi(g)\T = \T^{-1}  M_{\hat{a}(g^\sharp \cdot)} \,\T R_g = \Lambda_{\tilde{r}_g(a)} R_g. 
		\end{equation*} 

		Then, $R_g\Lambda_a = \Lambda_a R_g$ if and only if $\Lambda_{\tilde{r}_g(a)} R_g = \Lambda_a R_g$. Since $R_g$ is invertible, this holds if and only if		
		$\tilde{r}(g)(a) = a$ or, equivalently, $\wh{a}(g^{\sharp}\w)=\wh{a}(\w)$ for a.e. $\w\in \Omega.$
	\end{proof}
	
	If $a\in\ell_2(\ZZ)$ is a sequence of bounded spectrum which satisfies any of the conditions of Proposition \ref{prop:Lambda_a-G-preserving} and $V\subset L^2(\RR)$ is $\Gamma$-invariant, then  $\Lambda_a(V)\subseteq V$ and $\Lambda_a:V\to V$ is $\Gamma$-preserving.
	
	\begin{rem}
		If the group $G$ is infinite, often this class of operators is very small. Indeed, for any $s_0\in\ZZ$, by the invariance $\eqref{invariance-a-r_g}$ of Proposition \ref{prop:Lambda_a-G-preserving}, we have that
		$$\sum_{s\in\ZZ} |a(s)|^2 \geq \sum_{s\in\{g^{-1}s_0\,:\,g\in G\} }  |a(s)|^2 = \#\{g^{-1}s_0\,:\,g\in G\} .|a(s_0)|^2.$$
		Since the sequence $a$ is in $\ell_2(\ZZ)$, for every $s\in\ZZ$ where $a(s)\neq 0$ we have
		\begin{equation}\label{eq:finite-orbits}
			\#\{gs\,:\,g\in G\}<\infty.
		\end{equation}
		For example, consider the group of translations and shears in $\R^2$. That is, $\RR = \R^2$, $\ZZ = \Z^2$ and $G = \{\binom{1 \ k}{0 \ 1} \, : \, k \in \Z\}$, which preserves the lattice $\Z^2$. For $s=(s_1,s_2)\in\Z^2$ we have that $gs=(s_1+ks_2,s_2)$. Hence, if $s_2\neq 0$ then $\#\{gs\,:\,g\in G\}=\infty$ and so $a(s)=0$ necessarily. Thus, the operators of this kind must be of the form
		$$\Lambda_{a}=\sum_{s_1\in\Z} a(s_1,0) T_{(s_1,0)}.$$

		
		Furthermore, if $G$ were an infinite group acting faithfully over $\ZZ$, then every operator $\Lambda_a$ which commutes with $R_g$ for every $g\in G$ must satisfy that $a(s)=0$ for every $s\in\ZZ\setminus\{0\}$.
	\end{rem}

	\subsection{$\Gamma$-diagonalization}\label{sec:GD}

	Now, we aim to find conditions on $\ZZ$-preserving operators in order to obtain a $\ZZ$-diagonalization like in Definition \ref{def:s-diag} where each $\ZZ$-eigenvalue of the decomposition commute with the unitary representation $T_kR_g$ and each $\ZZ$-eigenspace is $\Gamma$-invariant.
		
	\begin{defi}
		Let $V\subset L^2(\RR)$ be a $\Gamma$-invariant space and $L:V\to V$ a bounded $\Gamma$-preserving operator. Let $a\in\ell_2(\ZZ)$ be of bounded spectrum, we say that $\Lambda_a:V\to V$ is a $\Gamma$-eigenvalue of $L$ if $\tilde{r}_g(a) = a$ for every $g\in G$ and $\Lambda_a$ is a $\ZZ$-eigenvalue of $L$, i.e.
		$$V_a := \ker (L-\Lambda_a)\neq \{0\}.$$
		
		Furthermore, we will say that $L$ is $\Gamma$-diagonalizable if it admits a $\ZZ$-diagonalization $\{a_j\}_{j\in I}$ of $L$ where $\Lambda_{a_j}$ is a $\Gamma$-eigenvalue for every $j\in I$.
	\end{defi}
	Observe that, in this case, the $\ZZ$-eigenspace $V_a$ associated to a $\Gamma$-eigenvalue $\Lambda_a$ is a $\Gamma$-invariant subspace of $V$. In particular, if $L$ is $\Gamma$-diagonalizable, then
	$$L = \sum_{j\in I} \Lambda_{a_j} P_{V_{a_j}}$$ where each $V_{a_j}$ is  $\Ga$-invariant subspace and the convergence of the series is in the strong operator topology sense.

	In what follows, we will assume that there exists a Borel set $\Omega_0\subset \Omega$ which is a \emph{transversal} for the action of $G$ on $\Omega$, that is, $\Omega_0$ intersects each orbit of the action of $G$ on $\Omega$ in exactly one point. We remark that a Borel transversal for the action of $G$ on $\Omega$ is not necessarily a tiling of $\Omega$, i.e. a set such that $\{g^\sharp\Omega_0\}_{g\in G}$ is an a.e. partition of $\Omega$.
	
	When $G$ is finite, the existence of such set is ensured by \cite[Theorem 12.16]{Kech1995}. Moreover, the existence of a Borel transversal is equivalent to the existence of a Borel \emph{selector} for the action of $G$ on $\Omega$ (see \cite{Kech1995}), that is, a Borel function $s:\Omega\to\Omega$ such that for every $\w,\w'\in \Omega$, we have that 
	$$\w'\in O_{G}(\w)\,\Rightarrow\, s(\w)=s(\w')\in O_{G}(\w),$$ where $$O_G(\w)=\{\w'\in\Omega\,:\,\w'=g^\sharp\w,\,g\in G\}.$$

	Under this assumption we are able to prove our final goal. In the next theorem we show that a normal $\ZZ$-diagonalizable operator which is $\Gamma$-preserving always admits a $\Gamma$-diagonalization.
	
	\begin{theo}\label{thm:s-diag-iff-gamma-diag}
		Let $V\subset L^2(\RR)$ be a $\Gamma$-invariant space and $L:V\to V$ a bounded normal $\Gamma$-preserving operator. Then, $L$ is $\Gamma$-diagonalizable if and only if it is $\ZZ$-diagonalizable.
	\end{theo}
	\begin{proof}
		We just need to prove that if $L$ is $\ZZ$-diagonalizable, then it admits a $\ZZ$-diagonalization conformed by $\Gamma$-eigenvalues. 
		Assume that $\{a_j\}_{j\in I}$ is a $\ZZ$-diagonalization of $L$. By Theorem \ref{thm:L-diag-O-diag}, $O(\w)$ is diagonalizable and $\sigma_p(O(\w))\subset\{\wh{a}_j(\w)\,:\,j\in I\}$ for a.e. $\w\in\Omega$.
		
	  	Now, let $s:\Omega\to\Omega$ be a Borel selector for the action of $G$ on $\Omega$. In particular, we have that $\sigma_p\left(O(s(\w))\right) \subset\{\wh{a}_j(s(\w))\,:\,j\in I\}$ for a.e. $\w \in\Omega$.
		Moreover, by \eqref{prop:spectrum-invariant-2} in Proposition \ref{prop:spectrum-invariant} we see that $\sigma_p(O(\w)) = \sigma_p\left(O(s(\w))\right)$ for a.e. $\w\in\Omega$. Thus, taking $\lambda_j(\w)=\wh{a}_j\circ s(\w)$ we get that
		\begin{equation}
			\sigma_p(O(\w))\subset\{\lambda_j(\w)\,:\,j\in I\}
		\end{equation} 
		for a.e. $\w\in\Omega$. Since $s$ is a selector, we see that $\lambda_j(g^\sharp\w)=\lambda_j(\w)$  for every $j\in I$, $g\in G$ and a.e. $\w\in\Omega$. Also, given that $s$ is a Borel selector and $\wh{a}_j\in L^\infty(\Omega)$, we obtain that $\lambda_j\in L^\infty(\Omega)$. 
		
		Since $O(\w)$ is diagonalizable the following orthogonal decomposition holds
		\begin{equation}\label{eq:J-decomp-gamma}
			\J(\w)=\underset{j\in I}{\bigpoplus} \ker(O(\w)-\lambda_{j}(\w)\mathcal I_{\w})
		\end{equation} 
		for a.e. $\w\in \Omega$.
		We now discard those functions $\lambda_j$ such that $\ker(O(\w)-\lambda_{j}(\w)\mathcal I_{\w})=\{0\}$ for a.e. $\w\in\Omega$. For each remaining $j$, there exists a sequence $b_j\in\ell_2(\ZZ)$ of bounded spectrum such that $\widehat{b}_j=\lambda_j$. So, $\Lambda_{b_j}$ is a $\Gamma$-eigenvalue of $L$ for every $j$ and by \eqref{eq:J-decomp-gamma} we conclude that 
		\begin{equation}
			V=\underset{j}{\bigpoplus} V_{b_j}.
		\end{equation}
		Hence, $L$ is $\Gamma$-diagonalizable.
		\end{proof}
	
{\bf Acknowledgements}: We thank the anonymous referees for the very careful reading of the manuscript and many positive suggestions which helped improve the paper.
	
	This project has received funding from the European Un\-ion's Horizon 2020 research and innovation programme under the Marie Sk\l odowska-Curie grant agreement No 777822. In addition, D. Barbieri and E. Hern\'andez were supported by Grants MTM2016-76566-P and PID2019-105599GB-I00 (Ministerio de Ciencia a Innovaci\'on, Spain). C. Cabrelli, D. Carbajal and U. Molter were supported by Grants UBACyT 20020170100430BA (University of Buenos Aires), PIP11220150100355 (CONICET) and PICT 2018-03399 (Ministery of Science and Technology from Argentina).

\end{document}